\documentclass[11pt]{article}
\usepackage{amsmath,amssymb,amsthm,graphicx,fixmath}

\usepackage{setspace}
\usepackage{ifthen}
\usepackage{tikz}

\def\Re{\mathbb R}

\def\epsilon{{\varepsilon}}
\providecommand{\remove}[1]{}

\theoremstyle{plain}
\newtheorem{theorem}{Theorem}[section]
\newtheorem{lemma}[theorem]{Lemma}

\newtheorem{conjecture}[theorem]{Conjecture}
\newtheorem{problem}[theorem]{Problem}

\theoremstyle{definition}
\newtheorem{definition}[theorem]{Definition}
\theoremstyle{remark}
\newtheorem{remark}[theorem]{Remark}

\newcommand{\E}{\mathcal{E}}

\newcommand{\F}{\mathcal{F}}

\def\famG{\mathcal{G}}

\newcommand*{\floor}[1]{\lfloor#1\rfloor}%
\newcommand{\cardin}[1]{\lvert {#1} \rvert}

\newcommand{\pth}[1]{\!\left({#1}\right)}

\begin{document}
	
	\title{A survey of Zarankiewicz problems in geometry}
	\author{%
		 Shakhar Smorodinsky
		\thanks{Department of Computer Science, Ben-Gurion University of the NEGEV, Be'er Sheva 84105, Israel.
				Research partially supported by the Israel Science Foundation (grant no.~1065/20) and by the United States -- Israel Binational Science Foundation (NSF-BSF grant no.~2022792).
			\texttt{shakhar@bgu.ac.il}}
	}

	\date{}
	\maketitle

 \begin{abstract}
    One of the central topics in extremal graph theory is the study of the function $ex(n,H)$, which represents the maximum number of edges a graph with $n$ vertices can have while avoiding a fixed graph $H$ as a subgraph. Tur{\'a}n provided a complete characterization for the case when $H$ is a complete graph on $r$ vertices. Erd{\H o}s, Stone, and Simonovits extended Tur{\'a}n's result to arbitrary graphs $H$ with $\chi(H) > 2$ (chromatic number greater than 2). However, determining the asymptotics of $ex(n, H)$ for bipartite graphs $H$ remains a widely open problem. A classical example of this is Zarankiewicz's problem, which asks for the asymptotics of $ex(n, K_{t,t})$.

    In this paper, we survey Zarankiewicz's problem, with a focus on graphs that arise from geometry. Incidence geometry, in particular, can be viewed as a manifestation of Zarankiewicz's problem in geometrically defined graphs.
\end{abstract}
		
	
	\section{Introduction} \label{sec:intro}
\subsection{Background}
This survey explores the Zarankiewicz problem, a central question in extremal graph theory concerning the maximum number of edges in a graph that avoids a given bipartite subgraph, specifically focusing on its rich connections to geometric settings.  The problem, while seemingly straightforward, has proven remarkably challenging. We begin by reviewing the foundational work on Turán's theorem and the Kővári-Sós-Turán bound, providing a solid base for understanding the problem's complexity in the general case.
The survey then delves into the many exciting advancements made in recent years, particularly concerning the Zarankiewicz problem for geometric intersection graphs.  We will examine improved bounds achieved under additional structural constraints, such as the restriction of induced subgraphs or the incorporation of VC-dimension. We will discuss significant results, including the strengthened Kővári-Sós-Turán bounds for graphs avoiding specific induced subgraphs or those with bounded VC-dimension.  We’ll also explore several examples illustrating how results in this area tie directly to problems in incidence geometry. These include point-line incidences (and the Szemerédi-Trotter Theorem), incidences between points and other geometric objects (e.g., discs, curves, or hyperplanes), and related applications.  We then consider the Zarankiewicz problem in the context of intersection graphs of geometric objects, highlighting recent progress and open questions. Finally, we briefly mention the more general version of Zarankiewicz problem for hypergraphs.

This survey aims to provide a comprehensive overview for researchers interested in extremal combinatorics and discrete geometry, particularly those working at the intersection of graph theory and geometry.  The ultimate goal is to provide a current state-of-the-art perspective for this active research area, highlighting its complexity and pointing toward future research directions.
\subsection{Preliminaries}
	A central research area in extremal combinatorics is \emph{Tur\'{a}n-type questions}, which ask for the maximum number of edges in a graph on $n$ vertices that does not contain a copy of some fixed graph $H$.
 \begin{definition}
     Let $H$ be a fixed graph and let $\famG$ denote some family of graphs. Denote by $ex_{\famG}(n,H)$ the maximum number of edges that a graph $G \in \famG$ on $n$ vertices can have if it does not contain $H$ as a subgraph (not necessarily induced).
     When $\famG$ is the family of all graphs we simply write $ex(n,H)$. 
 \end{definition}
 
 Such a research direction of understanding the function $ex(n,H)$ was initiated in 1941 by Tur\'{a}n, who proved the following classical result which is known as {\em Tur{\' a}n's Theorem}.
 \begin{theorem}[Tur{\' a}n's Theorem \cite{Turan41}]
 For $r \geq 2$
$$ ex(n,K_{r+1}) = \left(1-\frac{1}{r}+o(1)\right)\frac{n^2}{2}.$$
 \end{theorem}

\begin{remark}
  The bound for the special case of $r=2$, namely, triangle free graphs was obtained by Mantel already in 1907 \cite{Mantel}. The lower bound construction of such extremal graphs is obtained by considering the complete $r$-partite graph with $n$ vertices so that every pair of parts differ by at most one in its sizes.
 \end{remark}

Later, Erd\H{o}s, Stone and Simonovits \cite{ErdosStone} extended Tur{\' a}n's theorem for arbitrary graphs $H$ and proved a bound that depends on the chromatic number of $H$. 
\begin{theorem}\cite{ErdosStone}
\label{thm:Erdos-Stone}
    $$ex(n,H) = \left(1-\frac{1}{\chi(H)-1}+o(1)\right)\frac{n^2}{2}.$$
\end{theorem}

\medskip \noindent \textbf{Zarankiewicz's problem.} 
Notice that when $H=K_{r+1}$ so $\chi(H)=r+1$ Tur{\' a}n's Theorem is derived as a special case.	
Note also that when $\chi(H)=2$ namely, $H$ is bipartite then Theorem~\ref{thm:Erdos-Stone} only provides us with the 
implicit bound $ex(n,H)=o(n^2)$.  
Indeed, providing sharp explicit bounds for the bipartite case turned out to be significantly harder, and the question is still widely open for most bipartite graphs (see the survey~\cite{Sudakov10}). The case of $H$ being a complete bipartite graph was first studied by Zarankiewicz in 1951:

\begin{problem}[Zarankiewicz's problem] What is the asymptotic value of  $ex(n,K_{s,t})$? 
\end{problem}
	
In one of the cornerstone results of extremal graph theory, K{\H o}v\'{a}ri, S{\'o}s and Tur{\'a}n~\cite{KST54} proved the following upper bound:
 
 \begin{theorem}[K{\H o}v\'{a}ri, S{\'o}s and Tur{\'a}n \cite{KST54}] 
 \label{thm:KST}
  $$ex(n,K_{s,t})=O_s(n^{2-\frac{1}{t}}).$$
   \end{theorem} 
 
 \begin{remark}
 In fact, when the host graph is also a bipartite graph $G=(A \cup B,E)$ (with $|A|=m, |B|=n$) and $G$ avoids a $K_{s,t}$ as a subgraph  K{\H o}v\'{a}ri, S{\'o}s and Tur{\'a}n proved that the number of edges in $G$ is bounded by $$O_{s,t}\left(m\cdot n^{1-\frac{1}{s}} + n\cdot m^{1-\frac{1}{t}} + m + n\right).$$
 \end{remark}
 
     The general bound $O_s(n^{2-\frac{1}{t}})$ of Theorem~\ref{thm:KST} is asymptotically sharp when $t=2,3$ (see, e.g., \cite{reiman1958,Brown66}) or  when $s$ is sufficiently large with respect to $t$ (\cite{AlonRS99,Bukh24}). For $t=2$, let us describe a well known lower bound using the incidence graph of points and lines in a finite projective plane. A \textit{finite projective plane of order $q$} consists of a set $X$ of $q^2+q+1$ elements called \textit{points}, and a family $\mathcal{L}$ of $q^2+q+1$ subsets of $X$ called \textit{lines}, which satisfies the following properties:
\begin{enumerate}
	\item Each line has $q+1$ points.
	\item Any point belongs to exactly $q+1$ lines.
	\item Every two points lie on a unique line.
	\item Any two lines meet in a unique point.
\end{enumerate}
There is a well-known construction of a finite projective plane of order $q$, $PG(2,q)$, from a finite field $\mathbb{F}_q$ when $q$ is a prime power. For more details on finite projective planes, refer to \cite[Section 12.4]{jukna_book}.
Next, for a prime power $q$, let $X$ and $\mathcal{L}$ be the set of points and lines of $PG(2,q)$, respectively. We construct a bipartite graph $G=(V,E)$ on the vertex set $X\cup \mathcal{L}$ such that $p\in X$ and $l\in \mathcal{L}$ are adjacent in $G$ if and only if $p$ is on $l$. By (i) and (ii), $G$ is $(q+1)$-regular, so the number of vertices is $|V|=2(q^2+q+1)$ and the number of edges  is $(q+1)(q^2+q+1) =\Omega(|V|^{\frac{3}{2}})$. Also, $G$ does not contain $K_{2,2}$ as a subgraph by (iii) (or (iv)).

     Similarly, a lower bound construction for $t=3$ is an incidence graph of points and spheres (of a carefully chosen fixed radius) in a three-dimensional finite affine space. The question whether the K{\H o}v\'{a}ri-S{\'o}s-Tur{\'a}n theorem is tight for $t \geq 4$ is one of the central open problems in extremal graph theory.

Note that in order to obtain asymptotic upper bounds for Zarankiewicz's problem for arbitrary graphs, one can assume that the underlying graph on $n$ vertices avoiding $K_{s,t}$ is also bipartite.
Indeed, it is well known and an easy exercise to show that every graph $G=(V,E)$
contains a bipartite graph $G'=(V,E')$ with $\cardin{E'} \geq \frac{\cardin{E}}{2}$.
So if $G$ avoids a $K_{s,t}$ as a subgraph then obviously also $G'$ so any upper bound on $\cardin{E'}$ implies the same asymptotically upper bound (with a factor of $2$) on $\cardin{E}$.

\section{Strengthening of the K{\H o}v\'{a}ri, S{\'o}s and Tur{\'a}n bound}
Here we mention several cases for which improved bounds on the Zarankiewicz's problem are known when we add to the hypothesis on the host graph which avoids a $K_{s,t}$ some additional restrictions.
For simplicity of presentation, from now on we assume that $s=t$. 

\medskip \noindent \textbf{Zarankiewicz's problem for graphs avoiding an induced copy of a fixed bipartite graph $H$.} 

Let us start with the case where we add (to the assumption that the host graph does not contain a copy of $K_{t,t}$) the additional assumption that it does not contain an \textit{induced} copy of some fixed bipartite graph $H$. Problems involving graphs without an induced copy of a fixed graph $H$ are central in structural graph theory. The classical Erd{\H o}s-Hajnal conjecture is an example \cite{ErdosH89}.  Let $H$ be a fixed bipartite graph. We are interested in the maximum number of edges that a graph $G$ on $n$ vertices can have if it satisfies:
\begin{enumerate}
    \item $G$ does not contain a copy of $K_{t,t}$.
    \item $G$ does not contain an induced copy of $H$.
\end{enumerate}
Notice that it does not make any sense to omit the requirement that $G$ does not contain a copy of $K_{t,t}$ generally. Indeed, if $H$ is a bipartite graph which is missing at least one edge (i.e., $H$ is not a complete bipartite graph), then any complete bipartite graph $G$ does not contain an induced copy of $H$. It seems that for graphs $G$ with many edges, the only obstruction to contain an induced copy of a fixed bipartite graph $H$ is to contain a large complete bipartite graph. Hence, it makes sense to add the first requirement. Note also that the case where $H$ is not bipartite is not very interesting.
Indeed, if $H$ is not bipartite and $G$ does not contain an induced copy of $H$ then any bipartite (not necessarily induced) subgraph of $G$ also does not contain an induced copy of $H$ so by the remark above the problem is equivalent to the original Zarankiewicz problem (asymptotically). 

The above setting is related to the recently studied notion of \textit{degree boundedness}. See the survey \cite{Mcarthy2024survey} and the references therein for more on degree boundedness. 

The following theorem was discovered independently by Gir{\~{a}}o and Hunter (\cite{GH23} Lemma 7.1) and by Bourneuf, Bucić, Cook, and Davies (\cite{bourneuf2023} Theorem 1.4):

\begin{theorem}[\cite{bourneuf2023,GH23}]
    \label{thm:deg-boundedness}
    Let H be some fixed bipartite graph. Then there exists a number
$\epsilon= \epsilon(H) > 0$ so that for any integer $t$, every graph with $n$ vertices and no induced copy of $H$ that avoids also $K_{t,t}$ has at most $O_t(n^{2-\epsilon})$ edges.
Moreover, there exists a polynomial $f=f_H$ that depends on the graph $H$ such that the dependency on $t$ hidden in the big-`O' notation is bounded $f_H(t)$.
\end{theorem}

\begin{remark}
    Notice that when $t$ is much larger than $1/\epsilon$ (for the constant $\epsilon=\epsilon(H)$ of Theorem~\ref{thm:deg-boundedness}) then the exponent in the bound is much better than $2-\frac{1}{t}$ given in the K{\H o}v\'{a}ri, S{\'o}s and Tur{\'a}n theorem. Below, we use the theory of VC-dimension and provide a different and (to the best of our view) a simple short proof of  Theorem~\ref{thm:deg-boundedness} for the special case where the host graph $G$ is also bipartite.
\end{remark}

Obtaining better upper bounds on $K_{t,t}$-free graphs which do not contain an induced copy of any graph $H \in \F$ for some fixed family of graphs $\F$ were studied also for several special cases. For example, fix a graph $H$ and let $\F$ be the family of all \emph{proper} subdivisions of $H$. Those are all graphs that can be obtained by subdividing all the edges in $H$. Namely, replacing each such edge with some path of length at least $2$ such that all these paths are internally vertex disjoint.

\begin{theorem}[\cite{bourneuf2023}]
\label{thm:induced-subdivisions}
For every graph $H$, there exists a polynomial $f=f_H(t)$, such that if $G$ is a graph on $n$ vertices which avoids a $K_{t,t}$ and avoids all induced subgraphs each being isomorphic to some subdivision of $H$ then $G$ has at most $f(t)n$ edges. 
\end{theorem}

Gir{\~{a}}o and Hunter proved that when $H$ is a clique then in order to get the linear bound its enough to forbid only proper subdivisions of $H$ which are {\em balanced}. Meaning where all edges are subdivided to a path of the same length.

\begin{theorem}[\cite{GH23}]
\label{thm:induced-proper-subdivision}
Let $H=K_h$ be the clique on $h$ vertices. Let $G$ be a graph on $n$ vertices which avoids $K_{t,t}$. Assume that $G$ also avoids any induced subgraph which is isomorphic to some balanced proper subdivision of $H$.
Then $G$ has at most $2t^{500h^2}n$ edges.
\end{theorem}
Such linear bounds were obtained in several previous works, starting with the fundamental work of K\"{u}hn and Osthus~\cite{KuhnO04}. 

It makes sense to optimize the value $\epsilon(H)$ given in Theorem~\ref{thm:deg-boundedness}. In this regard,
Hunter, Milojević, Sudakov, and Tomon \cite{HMST24}
proposed the following  beautiful conjecture:

\begin{conjecture}
\label{conj:Sudakovetal}
        For every bipartite graph $H$ there is a function $f_H$ such that any graph with $n$ vertices that does not contain an induced copy of $H$ and avoids $K_{t,t}$ has at most $f_H(t)\cdot ex(n,H)$ edges.
\end{conjecture}

The conjecture is known to hold, for example, when $H$ is a tree (see, e.g., \cite{HMST24,KiersteadP94,ScottSS23}).

\medskip \noindent \textbf{Graphs with bounded VC-dimension.} 
The Vapnik-Chervonenkis dimension of a hypergraph is a measure of its complexity, which plays a central role in statistical learning, computational geometry, and other areas of computer science and combinatorics (see, e.g.,~\cite{AHW87,BEHW89,MV18}).
Many graphs and hypergraphs that arise in geometry have bounded VC-dimension.
 
 \begin{definition}[VC-dimension]
 \label{def:VC-dim}
 Let $H=(V,E)$ be a hypergraph. 
 The \emph{Vapnik-Chervonenkis (VC) dimension} of a hypergraph $H = (V, E)$ is the largest integer $d$ such that there exists a subset $S \subseteq V$ (not necessarily in $E$) with $|S| = d$ that is \emph{shattered} by $E$. A subset $S$ is said to be shattered by $E$ if, for every subset $T \subseteq S$, there exists a hyperedge $e \in E$ such that $e \cap S = T$. 
\end{definition}
 The dual hypergraph of a hypergraph $H=(V,E)$ is $H^*=(V^*,E^*)$, where $V^*=E$ and each $v \in V$ gives rise to  the hyperedge $e_v \in E^*$, where $e_v=\{  e \in E : v \in e \}$. 
 When the VC-dimension of $H$ is $d$, the VC-dimension of $H^* $ is denoted by $d^*$. 

 \begin{definition}
     The \emph{primal shatter function} of a hypergraph $H = (V, E)$ is the following function $\pi_H : \mathbb{N} \rightarrow \mathbb{N}$:
\[
\pi_H(m) = \max_{S \subseteq V, |S| = m} |\{ S \cap e : e \in E \}|.
\]
 The value $\pi_H(m)$ represents the maximum number of distinct subsets of a set $S$ of cardinality $m$ that can be realized as intersections with hyperedges in $E$.
 \end{definition}

The following lemma known as the Perles-Sauer-Shelah  lemma provides an upper bound on the shatter function for hypergraphs with bounded VC-dimension (See, e.g., \cite{MATOUSEK}):
\begin{lemma}[{\bf Sauer-Shelah-Perles}]
    \label{Lem:shatter-function}
    Let $H = (V, E)$ be a hypergraph with VC dimension $d$. 
\[
\pi_H(m) \leq \sum_{i=0}^{d} \binom{m}{i}.
\]
In particular, if $m > d$, then $\pi_H(m) \leq m^d$.
\end{lemma}

	Any graph $G=(V,E)$ defines a natural structure referred to as the neighborhood hypergraph. This is the hypergraph $H=(V,\E)$ on the same set of vertices where the hyperedges are all neighborhoods of vertices in $G$. That is, $\E=\{N(v)| v \in V\}$ where for $v \in V$ $N(v)= \{u \in V| \{v,u\} \in E\}$.
 For the special case of a bipartite graph $G=G_{A,B}$ with vertex set $V(G)=A \cup B$ and edge set $E(G) \subset A \times B$, we define two hypergraphs: the primal hypergraph $H_G=(A,\E_B)$, where $\E_B=\{  N(b):b \in B  \}$ is the collection of the open neighborhoods of the vertices in $B$, and the dual hypergraph $H_G^*=(B,\E_A)$, defined similarly.
 The VC-dimension of $G$ is defined as the VC-dimension of $H_G$.
 For the special case of when $G$ is bipartite we distinguish between the VC-dimension of $H_G$ and the dual VC-dimension of $G$ that is defined as the VC-dimension of $H_G^*$. Similarly, we define the shatter function $\pi_G$, as the shatter functions of $H_G$ (and for bipartite graphs  the dual shatter function $\pi_G^*$ of $G$ as the shatter function of $H_G^*$). Many well-studied problems in graph theory become more
manageable if the VC-dimension is bounded. Some of those problems were
studied in e.g., \cite{FoxPS21,FoxPS23,NguyenSS24}.
	
	In a remarkable result, Fox, Pach, Sheffer, Suk and Zahl~\cite{FPSSZ17} improved the bound of the K{\H o}v\'{a}ri-S{\'o}s-Tur{\'a}n theorem for graphs with VC-dimension at most $d$ (for $d<t$). Their theorem is stated in terms of the constants $d$ and $d^*$ of the exponents in the primal and dual shatter functions. They showed:
	\begin{theorem}[\cite{FPSSZ17}]
		\label{thm:main-boundedvc}
		Let $t \geq 2$ and let $G_{A,B}$ be a bipartite graph with $|A|=m$ and $|B|=n$, satisfying $\pi_G(\ell)=O(\ell^d)$ and $\pi_G^*(\ell)=O(\ell^{d^*})$ for all $\ell$. If $G$ is $K_{t,t}$-free, then 
  \begin{itemize}
        \item 
      $$ |E(G)|=O_{t,d,d^*}(\min \{  mn^{1-\frac{1}{d}}+n,  nm^{1-\frac{1}{d^*}}+m  \}).$$
      \item In particular, if $m=n$ and $d^*=d$ then $$|E(G)|=O_{t,d}(n^{2-\frac{1}{d}}).$$
  \end{itemize}
	\end{theorem}
 
Note that as in the case of Theorem~\ref{thm:deg-boundedness}, the bound $O_t(n^{2-\frac{1}{d}})$ is much better than the bound $O(n^{2-\frac{1}{t}})$ of K{\H o}v\'{a}ri-S{\'o}s-Tur{\'a}n whenever $t > d$.

\medskip \noindent \textbf{No induced $H$ versus bounded VC-dimension when the host graph is bipartite.} 
In what follows, we argue that Theorem~\ref{thm:main-boundedvc} and Theorem~\ref{thm:deg-boundedness} are equivalent in some sense in the special case when the underlying host graph is bipartite. In that case, we show that Theorem~\ref{thm:main-boundedvc} can be derived from Theorem~\ref{thm:deg-boundedness} albeit with a weaker constant in the exponent. We also show that Theorem~\ref{thm:deg-boundedness} can be derived from Theorem~\ref{thm:main-boundedvc} with an explicit constant $2-\frac{1}{\frac{|V(H)|}{2}+\log |V(H)|-1}$ in the exponent where $|V(H)|$ is the number of vertices of the forbidden induced graph $H$. 

The first implication follows easily from the following theorem:
\begin{theorem}
    Let $d> 0$ be a fixed integer. There exists a bipartite graph $H=H(d)$ such that  any graph $G=(V,E)$ with shatter function satisfying $\pi_G(\ell) \leq C\cdot{\ell}^d$ does not contain an induced copy of $H$.
\end{theorem}

\begin{proof}
The proof is easy (and is given implicitly in e.g., \cite{BousquetLLPT15}, Lemma 3.3).
    Let $\ell=\ell(d)$ be the least integer for which $2^{\ell} > C\cdot{\ell}^d$.
    Notice that such $\ell$ exists and it is easily verified that $\ell = O(d \log d)$.
    We construct the following bipartite graph $H=H(d)$ which can be thought of the vertex-hyperedge incidence relation of a complete hypergraph on  $\ell$ vertices. Namely, let $H$ be the bipartite graph on two sets $S$ and $T$. $S$ has $\ell$ vertices and $T$ has $2^{\ell}$ vertices indexed by the power set of $S$. That is, for every subset $W \subset S$ there is a unique vertex $v_{W} \in T$. Each such vertex $v_W$ is connected to the subset of vertices in $W \subset S$ that it represents. It is easily verified that $G$ cannot contain an induced copy of $H$. Indeed, assume to the contrary that $G$ contains an induced copy of $H$. We claim that it implies that $S$ is shattered in the neighborhood hypergraph of $G$. Indeed, for every subset $W \subset S$ by construction of $H$ there exists a distinct vertex $v_W \in T$ with the neighbors set $N_H(v_W)=S$. So in the induced copy of $H$ in $G$ there is a vertex $v$ such that $N_G(v)\cap S = W$. Since this holds for any set $W \subset S$ we get that $S$ is shattered. So in the neighborhood hypergraph $(V,\E)$ (of the graph $G=(V,E)$) we have $|\{S \cap e \mid   e\in {\cal E}\}|= 2^\ell$. On the other hand we have  $|\{S \cap e \mid   e\in {\cal E} \}| \leq \pi_G(\ell) \leq C\cdot{\ell}^d$. A contradiction to our choice of $\ell$. This completes the proof.
\end{proof}

Theorem~\ref{thm:deg-boundedness} only provides the bound $O_t(n^{2-\epsilon_H})$ for arbitrary large values of $t$ while Theorem~\ref{thm:main-boundedvc} provides a bound with the improved explicit constant $\epsilon_H=\frac{1}{d}$.

The next lemma proves that no containing no induced copy of $H$ implies a bounded VC-dimension on the hosting (bipartite) graph.

\begin{lemma}
\label{Lem:no-inuced-impliesVC}
Let $H=(S \cup T,U)$ be a fixed bipartite graph with parts $S$ and $T$ and edge set $U \subset S \times T$.
Put $s=|S|, t=|T|$ and assume without loss of generality that $t \geq s$. Let $d= s+t$ denote the total number of vertices in $H$. Let $G=(A\cup B , E)$ be a bipartite graph with parts $A$ and $B$.
If $G$ does not contain an induced copy of $H$ then both the VC-dimension of $G$ and the dual VC-dimension of $G$ is at most $s+{\log t}-1 \leq \frac{d}{2}+{\log d} -1$.
\end{lemma}

\begin{proof}
   For the set of vertices $S \cup T$ of $H$ put $S=\{x_1,...,x_s\}$ and $T = \{y_1,...,y_t\}$. Put $W=\{w_1,..,w_{\log t}\}$. Let us index the vertices of $T$ by the power set of $W$. Namely, for every $y \in T$ there is a unique subset $W_y \subset W$ that corresponds to $y$. Assume to the contrary that $G$ contains a shattered subset $Z$ of $k=s+{\log t}$ vertices. Notice that $Z$ must have all of its vertices belong to the same part in $G$ (either $Z \subseteq A$ or $Z \subseteq B$). Indeed, if a vertex $x\in Z$ is in $A$ and $y \in Z$ is in $B$ then consider e.g., the vertex $v=v_{\{x,y\}}$ witnessing the fact that $N(v)\cap Z=\{x,y\}$ for some $v$. Such a vertex is a neighbor of both $x$ and $y$ contradicting the fact that $G$ is bipartite. So assume without loss of generality that $Z \subset A$. For simplicity, let us abuse the notations and denote the vertices of $Z$ by $Z=S\cup W=\{x_1,...,x_s,w_1,...,w_{\log t}\}$. For every vertex $y \in T$ in the graph $H$ let $W_y \subset W$ be the subset of $W$ representing $y$. Let $S_y = N(y) \subset S$ be the set of neighbors of $y$ in $H$. Since $Z$ is shattered in $G$, it means in particular, that for the subset $S_y\cup W_y \subset Z$ there must exist a distinct vertex $v_y \in B$ in $G$  so that  $N(v_y) \cap Z = S_y\cup W_y$ where here $S_y$ refers to the image in $Z$ of $S_y$ and $W_y$ is the distinct subset of $W$ corresponding to $y$. It is easily verified that the induced subgraph of $G$ $G[S\cup \{v_y \mid y \in T\}]$ on the vertices $S\cup \{v_y \mid y \in T\}$ is an induced isomorphic copy of $H$ in $G$, a contradiction. Hence, there cannot be a shattered set in $A$ of size $k$. The same arguments hold almost verbatim for the case  $Z \subset B$. Hence, every shattered set has size at most $k-1$. This completes the proof of the lemma. 
\end{proof}
\begin{remark}
    The assumption in Lemma~\ref{Lem:no-inuced-impliesVC} that the hosting graph  $G$ is bipartite cannot be relaxed. Indeed, the following easy construction communicated to me by \cite{Yuditsky2015} shows that for any bipartite graph $H$ with at least five vertices and for any integer $d\geq 5$ there exists a graph $G$ with VC-dimension at least $d$ without an induced copy of $H$. To construct such a graph $G$, let the vertices of $G$ be the disjoint union of two sets $A$ and $B$ where $|A|=d$ and $|B|=2^d$ and for every subset $S \subset A$  there is a unique vertex $v \in B$ such that its neighbors in $A$ $N(v) \cap A$ is exactly the set $S$. Add edges to $G$ so that $A$ is a clique and $B$ is a clique. It is easily seen that $A$ is shattered in $G$ so the VC-dimension of $G$ is at least $d$. However, note that $G$ cannot contain an induced copy of $H$ as otherwise there would be a subset of size at least three in $H$ whose image in $G$ is contained in either $A$ or $B$. Since  both $A$ and $B$ are cliques it would imply that $H$ contains a triangle, a contradiction.
\end{remark}

Lemma~\ref{Lem:no-inuced-impliesVC} together with the Sauer-Shelah-Perles Lemma (Lemma~\ref{Lem:shatter-function}) combined with Theorem~\ref{thm:main-boundedvc} implies the following bound:
\begin{theorem}
\label{thm:no-induced}
    Let $H$ be a fixed bipartite graph with $d$ vertices and let $t$ be some fixed integer.
    Then any bipartite graph $G$ on $n$ vertices that does not contain a copy of $K_{t,t}$ and does not contain an induced copy of $H$ has $O_t(n^{2-\frac{1}{d/2+{\log d}-1}})$ edges.
\end{theorem}

\medskip \noindent \textbf{Further improvements for the Zarankiewicz problem for graphs with  bounded VC-dimension.} 

 Theorem~\ref{thm:main-boundedvc} spawned several follow-up papers. Janzer and Pohoata~\cite{JP20} obtained an improved bound of $o(n^{2-\frac{1}{d}})$ for bipartite graphs with VC-dimension $d$, where $m = n$ and $t \geq d>2$, using the hypergraph removal lemma~\cite{Gowers07}. Do~\cite{Do19} and Frankl and Kupavskii~\cite{FK21} obtained improved bounds when $t$ tends to infinity with $n$.

 Recently, Keller and Smorodinsky obtained a very short proof of the bound in Theorem~\ref{thm:main-boundedvc}  by using the recently introduced notion of $\epsilon-t$-nets \cite{KellerS24}.
 Let us first introduce the definition of the standard notion of $\epsilon$-nets for hypergraphs.
 \begin{definition}[$\epsilon$-nets]
Let $H=(V,E)$ be a hypergraph and $\epsilon>0$ be a real.
     An \emph{$\epsilon$-net} for $H$ is a set $S \subset V$ such that any hyperedge $e \in E$ with $|e| \geq \epsilon |V|$ contains a vertex from $S$. Namely, an $\epsilon$-net for $H$ is a hitting set for all hyperedges of cardinality at least $\geq \epsilon |V|$.
    \end{definition} 
The notion of $\epsilon$-nets was introduced by Haussler and Welzl~\cite{hw-ensrq-87} who proved that any finite hypergraph with VC-dimension $d$ admits an $\epsilon$-net of size $O((d/\epsilon)\log(d/\epsilon))$ (a bound that was later improved to $O((d/\epsilon)\log(1/\epsilon))$ in~\cite{KPW92}). $\epsilon$-nets were studied extensively and have found applications in diverse areas of computer science,
	including machine learning, algorithms, computational geometry, and social choice (see, e.g.,~\cite{ABKKW06,AFM18,BEHW89,Chan18}).

The following notion of \emph{$\epsilon$-$t$-nets} was introduced recently by Alon et al.~\cite{AJKSY22}. It generalizes the of $\epsilon$-nets\footnote{ And also generalizing the notion of \emph{$\epsilon$-Mnets} that was studied by Mustafa and Ray~\cite{MustafaR17} and by Dutta et al.~\cite{DuttaGJM19}}:
 
\begin{definition}[$\epsilon-t$-nets]
Let $H=(V,E)$ be a hypergraph, $\epsilon>0$ be a real and $t > 0$ an integer.
         A set $N \subset \binom{V}{t}$ of $t$-tuples of vertices is called an \emph{$\epsilon$-$t$-net} if any hyperedge $e \in E$ with $|e| \geq \epsilon n$ contains at least one of the t-tuples in $N$.
\end{definition} 

	Note that for $t=1$ an $\epsilon-1$-net is equivalent to the standard $\epsilon$-net.

	Alon et al. \cite{AJKSY22} proved the following theorem which extends the upper-bound of Haussler and Welzl~\cite{hw-ensrq-87} for sufficiently large $\epsilon$:
 
	\begin{theorem}\label{thm:eps-t-net}
		For every $\epsilon \in (0,1)$, $C>0$, and $t,d,d^* \in \mathbb{N} \setminus \{0\}$, there exists $C_1=C_1(C,d^*)$ such that the following holds. Let $H$ be a hypergraph on at least $C_1((t-1)/\epsilon)^{d^*}$vertices with VC-dimension $d$ and dual shatter function $\pi^*_H(m) \leq C \cdot m^{d^*}$. Then $H$ admits an $\epsilon$-$t$-net of size $O((d(1+\log t)/\epsilon)\log(1/\epsilon))$, all of which elements are pairwise disjoint.	 
	\end{theorem}

 Let us demonstrate how Theorem~\ref{thm:eps-t-net} yields a strikingly simple proof of the $O(n^{2-\frac{1}{d}})$ bound in Theorem~\ref{thm:main-boundedvc} of Fox et al.\cite{FPSSZ17}. For simplicity we only sketch the proof for the case $d^*=d$ and $m=n$. Later, in Section~\ref{subsec:pseudo-discs} we show how the ideas of using $\epsilon-t$-nets can be further exploited to obtain linear bounds on the Zarankiewicz's problem for bipartite intersection graphs of so-called pseudo-discs families.
	
	\begin{proof}
		Put $\epsilon = \frac{C_1^{1/{d}}(t-1)}{n^{1/{d}}}$, where $C_1$ is the constant from Theorem \ref{thm:eps-t-net}. Let $N$ be an $\epsilon$-$t$-net for $H_G$ of size $O((d(1+\log t)/\epsilon)\log(1/\epsilon))$, whose existence follows from Theorem~\ref{thm:eps-t-net}.
		
		Let $B'\subset B$ be the set of vertices with degree at least $\epsilon n = \Theta_{d,t} (n^{1-\frac{1}{d}})$ in $G$. We claim that 
		\[
		|B'| \leq (t-1)|N| =O_{d,t}( \frac{1}{\epsilon}\log \frac{1}{\epsilon})=O_{d,t}(n^{\frac{1}{d} }\log n).
		\]
		Indeed, on the one hand, for each $b \in B'$, the hyperedge $e_b$ contains a $t$-tuple from $N$. On the other hand, as $G_{A,B}$ is $K_{t,t}$-free, any $t$-tuple in $N$ participates in at most $t-1$ hyperedges of $H_G$. Thus, $|B'| \leq (t-1)|N|$, as asserted.
		
		To complete the proof, we note that $|E(G)| = \sum_{b\in B} d(b)$, where $d(b)$ is the degree of $b$ in $G$. Hence, we have
		$$
		|E(G)|=\sum_{b \in B} d(b) = \sum_{b \in B'}d(b) + \sum_{b \in B \setminus B'} d(b) \leq |B'|n+ |B \setminus B'|\epsilon n$$
        $$= O_{d,t}\left(n^{1+1/d}\log n + n^{2-1/d}\right) =O_{d,t}(n^{2-\frac{1}{d}}). 
		$$ for $d > 2$.
	\end{proof}

It remains an open problem already for the case $m=n$ to get an explicit improved asymptotic upper bound on the number of edges that any bipartite graph with VC-dimension $d$ can have if it excludes $K_{t,t}$:

\begin{problem}
    Let $t \geq d > 2$ be two fixed integers and let  $\famG$ be the family of all graphs with VC-dimension $d$. Obtain improved asymptotic bounds on $ex_{\famG}(n,K_{t,t})$.
\end{problem}

In what follows, we argue that it would be extremely challenging to improve the Fox et al. exponent in the upper bound $ex_{\famG}(n,K_{t,t})= O_{t,d}(n^{2-\frac{1}{d}})$ beyond the constant $2-\frac{1}{d-\omega(\log d)}$. That follows from the following lemma that implies that any bipartite graph that avoids $K_{t,t}$ as a subgraph has primal shatter function $\pi_G(\ell) = O({\ell}^d)$ for $d \leq t+\log t$. This also holds for the dual shatter function. To show this, it is enough to bound the VC-dimension of the primal hypergraph (the proof for the dual hypergraph is almost verbatim):

\begin{lemma}
    Let $G=(A\cup B,E)$ be a bipartite graph with vertices parts $A$ and $B$. Assume that $G$ avoids the complete bipartite graph $K_{t,t}$ as a subgraph. Then the VC-dimension of $H_G$ is at most $t+\log t-1$. 
\end{lemma}

\begin{proof}
The proof is similar to the proof of Lemma~\ref{Lem:no-inuced-impliesVC} above. Assume to the contrary that there exists a shattered set $S \subset A$ with cardinality $t+\log t$. Let us denote the vertices of $S$ by $S=\{v_1,v_2,\ldots,v_t,\ldots,v_{t+\log t}\}$. And let $T \subset S$ be the set $T=\{v_1,\ldots,v_t\}$ Consider the family of subsets of $S$ of the form $\{T \cup X| X \subset \{v_{t+1},\ldots,v_{t+\log t}\} \}$. By the fact that $S$ is shattered, we have that for every such $T \cup X$ there exists a vertex $v_{X} \in B$ whose neighborhood in $A\cap S$ is exactly $T \cup X$. In particular, every vertex in $T$ is a neighbor of $v_X$. The number of such vertices $v_X$ is $t$. This gives rise to a $K_{t,t}$ in $G$, a contradiction. This completes the proof. 
\end{proof}

\section{Zarankiewicz's problem for semi-algebraic graphs}
\label{sec:semi-algebraic}
A semi-algebraic graph is a bipartite graph that satisfies several algebraic constraints. It captures most of the settings that arise in the study of incidence geometry discussed later in Section~\ref{sec:incidences}.

A bipartite graph $G$ is semi-algebraic in $\Re^d$
if its vertices are represented by point
sets $P, Q \subset  \Re^d$ and its edges are defined as pairs of points $(p, q) \in P \times Q$ that satisfy
a Boolean combination of a fixed constant number of polynomial equations and inequalities in
$2d$ coordinates. 

 More generally, one can define a semi-algebraic bipartite graph in $(\Re^{d_1},\Re^{d_2})$. Namely, on two sets of points lying in two distinct dimensions $d_1$ and $d_2$ and has description complexity $s$.
\begin{definition}
    A bipartite graph $G=(P \cup Q,E)$ is called \emph{semi-algebraic with description complexity $s$} if $P$ is represented as a set of point in $\Re^{d_1}$ and $Q$ is represented as a set of points in $\Re^{d_2}$ and there are $s$ polynomials $f_1,f_2,\ldots,f_s \in \Re[x_1,\ldots,x_{d_1+d_2}]$ in $d_1+d_2$ variables and each of degree at most $s$ and a booloean function $\Phi(X_1,\ldots,X_s)$
such that for $(p,q)\in \Re^{d_1+d_2}$ we have:
$$
(p,q) \in E \Longleftrightarrow \Phi(f_1(p,q)\geq 0,\ldots,f_s(p,q)\geq 0)=1
$$
\end{definition}

 This definition allows for complex interactions between vertices, encompassing various types of geometric configurations beyond simple linear or circular arrangements. It captures many of the well-studied incidence problems
in combinatorial geometry (see, e.g., \cite{Pach2005}).

In their seminal paper Fox et al.~\cite{FPSSZ17} obtained a far-reaching generalization of the famous Szemer\'{e}di-Trotter bounds for points-lines incidences, which we discuss later in Section~\ref{sec:incidences}. They prove the following theorem:

\begin{theorem}
\label{Thm:Foxetal}
     Let $G = (P \cup Q, E)$ be a semi-algebraic bipartite graph on $n$ vertices in $\Re^d$ with description complexity $s$.
     If $G$ is $K_{t,t}$-free then:
     \begin{itemize}
         \item if $d=2$ then $|E|= O(n^{4/3})$
         \item if $d\geq 3$ then $|E|=O(n^{2-\frac{2}{d+1}+ \epsilon})$ for any $\epsilon > 0$ where the constant in the big-`O'  notation depends on $t,d,s, \epsilon$
     \end{itemize}

\end{theorem}
\begin{remark}
    In fact, Theorem~\ref{Thm:Foxetal} provides a more refined bound that generalizes to two sets where $P$ is in $\Re^{d_1}$ and $Q$ is in $\Re^{d_2}$ for arbitrary two dimensions $d_1, d_2$. For simplicity of presentation we only described the bound when $d_1=d_2=d$. 
\end{remark}
Let us demonstrate how Theorem~\ref{Thm:Foxetal} can be applied to, say, bipartite intersection graphs of discs. 
Let $G=(B\cup R,E)$ be a graph where $B$ is a family of $n$ ``blue" discs in the plane and $R$ is a family of $n$ red of discs in the plane and $E$ consists of all blue-red intersecting pairs. That is $E=\{ \{b,r\} | b \in B, r \in R, b \cap r \neq \emptyset\}$. Assume further that $G$ is $K_{t,t}$-free.
One can easily see that this graph is semi-algebraic in $(\Re^3,\Re^3)$ with description complexity $2$.
Indeed, each discs $d \in B \cup R$ with center $(x,y)$ and radius $r$ can be represented as the point $(x,y,r) \in \Re^3$.
Moreover, a disc $d_1=(x_1,y_1,r_1)$ has a non-empty intersection with another disc $d_2=(x_2,y_2,r_2)$ if and only if the following holds:
$$
(r_1+r_2)^2 -(x_1-x_2)^2 - (y_1-y_2)^2 \geq 0
$$ 
so in the terminology of semi-algebraic graph let  $f=f(x_1,y_1,r_1,x_2,y_2,r_2)$ be the ($6$-variate) polynomial of degree $2$ defined by $f(x_1,y_1,r_1,x_2,y_2,r_2)=(r_1+r_2)^2 -(x_1-x_2)^2 - (y_1-y_2)^2$ and let $\Phi(X)$ be the boolean function (with one variable) defined by  $\Phi(X)=X$.
So for two discs $b \in B$ and $r \in R$ we have that $\{b,r\} \in E$ if and only if the two points
$\{p,q\}$ representing them satisfy $\Phi(f(p,q) \geq 0)=1$.
Hence, $G$ is a semi-algebraic graph in $(\Re^3,\Re^3)$
with description complexity $2$ and $K_{t,t}$-free so by Theorem~\ref{Thm:Foxetal} $|E|=O_t(n^{\frac{3}{2}+\epsilon})$ for any $\epsilon > 0 $. This upper bound should be contrasted with the upper bound $O(n^{2-\frac{1}{t}})$ given by the K{\H o}v\'{a}ri, S{\'o}s, Tur{\'a}n theorem mentioned above. In Section~\ref{subsec:pseudo-discs} we show how to improve this bound to optimal $O(n)$ even for the more general case of an intersection graph of two families of pseudo-discs.

\section{Incidence Geometry}	
\label{sec:incidences}
	 An \emph{incidence graph} is a bipartite graph whose vertex set is a union of a set of points and a set of geometric objects, where the edges connect points to objects to which they are incident. Problems on incidence graphs are central in combinatorial and computational geometry. For example, the classical Erd\H{o}s' unit distances problem asks for an upper bound on the maximum pairs of points at Euclidean distance exactly $1$ among $n$ points in the plane. As is detailed below, this problem can be reduced to the study of the maximum number of edges that a certain incidence graph between $n$ points and $n$ unit circle in the plane can have. We start with the following yet another classical problem of Erd{\H o}s on incidences between points and lines in the plane. This problem asks for a bound on the maximum number of incidences (i.e., the number of times a point lies on a line) between a set of points and a set of lines in the Euclidean plane.
  For a finite set $P$ of points in the plane and a finite set $L$ of lines put $I(P,L)=\cardin{\{ (p,\ell) \mid p \in P , \ell\in L\}}$. Namely, $I(P,L)$ denotes the number of point-line incidences among $P$ and $L$. Put $I(m,n)= \max\{I(p,L) \mid |P|=m, |L|=n\}$. 
  Erd{\H o}s asked what is the asymptotic of $I(m,n)$. When $m=n$ we abuse the notation and denote it $I(n)$.
  For $P$ and $L$ as above, one can define the graph $G_{P,L}$ to be the bipartite graph which vertices is $P \cup L$ and $E$ the the set of all incident pairs $(p,\ell)$. That is $E = \{(p,\ell) \mid p \in P, \ell \in L, p \in \ell\}$. Note that bounding $I(m,n)$ is equivalent to bounding the maximum number of edges that such a graph can have. 
  Notice  also that the incidence graph $G_{P,L}$ is $K_{2,2}$-free since there can be at most one line that is incident to two given points. Hence, this is a special case of Zarankiewicz's problem for the class of point-line incidence graphs in the plane. By Theorem~\ref{thm:KST} we get the upper bound $I(n)=O(n^{3/2})$. 
	Erd{\H o}s provided a lower bound of the form $I(n) = \Omega(n^{4/3})$ (see details below) and conjectured that this bound is asymptotically sharp.

\medskip \noindent \textbf{Szemerédi-Trotter Theorem.}
The Szemerédi-Trotter theorem, published in 1983 \cite{SzemerediTrotter83}  provides the following asymptotically sharp bound for point-line incidences:
\[
I(m, n) = O\left(m^{2/3}n^{2/3} + m + n \right).
\]
So for the case $m=n$ we have $I(n)=O(n^{4/3})$.
This theorem is perhaps the starting point of the rich and diverse area of research known as \textit{incidence geometry}.

\medskip \noindent \textbf{Székely's Proof for the Szemerédi-Trotter bound:}
\newline
In 1997 Székely (see \cite{Szekely97}) obtained a short and beautiful proof of the Szemerédi-Trotter bound using the following so-called \textit{Crossing Lemma} for graphs. The Lemma was first proved by Ajtai, Chv{\'a}tal, Newborn and Szemer{\'e}di \cite{Ajtai1982CrossingFreeS} and by Leighton \cite{Leighton1983}.
 
 \begin{theorem}[Crossing Lemma \cite{Ajtai1982CrossingFreeS,Leighton1983}]
 \label{thm:thm:crossing-lemma}
Let \( G=(V,E) \) be a graph drawn in the plane. Assume that $|E| \geq 4|V|$.
Then the number of edge crossings \( \text{cr}(G) \) (i.e., pairs of edges on four distinct vertices that cross in the drawing), satisfies the following inequality:
\[
\text{cr}(G) \geq \frac{|E|^3}{64 |V|^2},
\]
\end{theorem}

Let us describe the proof of Székely (for the point-line incidences bound):
Let $P$ be a set of $m$ points and $L$ a set of $n$ distinct lines in the plane and let $I=I(P,L)$ denote the number of their incidences. Assume, without loss of generality that every line of $L$ contains at least one point of $P$.
We consider the following graph $G=(P,E)$ (together with an embedding in the plane) on the set $P$. For any pair of points $p,q \in P$ $\{p,q\} \in E$ if and only if the unique line passing through $p$ and $q$ belongs to $L$ and $p$ and $q$ are consecutive on that line. For such an edge we draw it using the straight line segment through $p$ and $q$.

For each line $\ell \in L$ that passes through at least one point of $P$ denote by $m_{\ell}$  the number of points of $P$ on $\ell$ and by $e_{\ell}$ the number of edges drawn on $\ell$. Obviously, for each $\ell \in L$ we have $e_{\ell} = m_{\ell}-1$. Moreover,
$$
|E|=\sum_{\ell \in L}e_{\ell} = (\sum_{\ell \in L} m_{\ell})-n = I-n.
$$ So $I=|E|+n $. Hence, it is enough to bound $|E|$. 
The key insight is to bound the number of edge crossings in $G$. The Crossing  Lemma provides a lower bound on the number of pairs of edges (in the drawing) of $G$. 
On the other hand, we can upper bound the number of crossings. An edge crossing  can be charged to the two (intersecting) lines they contain them. But two lines that cross can give rise to at most one crossing of $G$.  So the number of crossings is bounded from above by the number of pairs of lines which is $n^2$.
\[
cr(G) \leq n^2.
\]

Combining both the lower and upper bounds on $cr(G)$, we have if $|E|\geq 4m$:
\[
\frac{|E|^3}{64m^2} \leq n^2.
\]
Rearranging this inequality gives:
\[
|E| = O\left(m^{2/3}n^{2/3}\right).
\]
Taking into account the case $|E|< 4m$ we have 
$$|E| = O(m^{2/3}n^{2/3}+m).$$

Since the number of incidences satisfy $I = |E|+n$  we have:
$$
I = O(m^{2/3}n^{2/3} + m + n),$$ as required. \qed

A recent powerful technique that has significantly advanced the study of incidence problems is {\bf polynomial partitioning}. Introduced by Guth and Katz in their groundbreaking work on the Erdős distinct distances problem \cite{GuthKatz2015}, this method involves partitioning the space into cells defined by low-degree polynomials. A key result states that for a set of $n$ points in $\Re^d$, and for any $0 < D$  there exists a $d$-variate polynomial $f$ of degree $D$ such that the zero set $Z(f)$ of $f$ partitions $\Re^d$
into $O(D^d)$ cells (i.e., connected components of $\Re^d \setminus Z(f)$), each containing in its interior at most $\frac{n}{D^d}$ points. By carefully analyzing the incidences within each cell and across the boundaries, polynomial partitioning has led to improved bounds on the total number of incidences in a wide range of geometric settings. See, e.g., \cite{Sheffer2023}. This technique has demonstrated its versatility and effectiveness in understanding the interplay between combinatorics and geometry. 

\noindent \medskip \textbf{Proof (Using Polynomial Partitioning):}
We briefly sketch this technique to show yet another proof obtained first in \cite{KaplanMS12} of the Szemerédi-Trotter bound.
The idea is to use the partitioning properties of algebraic curves to divide the problem into manageable subproblems. For simplicity we only discuss the case $m=n$.
We need the following easy and weaker upper bound $I(P,L) = O(|P|^2+|L|)$. Denote by $L' \subset L$ the subset of lines that contain at least two points of $P$. Note that any point $p \in P$ can be incident to at most $|P|-1$ such lines in $L'$. So in total the number of incidences is bounded by $|P|\times (|P|-1) + |L \setminus L'| \leq |P|^2+|L|$ where the second term bounds the number of incidences coming from lines containing at most one point from $P$.

\noindent \medskip \textbf{Step 1: Partitioning Space with a Polynomial:}
Let \( P \) be a set of \( n \) points in \( \mathbb{R}^2 \), and \( L \) be a set of \( n \) lines. By the polynomial partitioning theorem, for any integer \( D > 0 \), we can find a non-zero polynomial \( f(x,y) \) of degree at most \( D \) that partitions the plane into \( O(D^2) \) cells, such that each cell contains at most \( O(n/D^2) \) points from \( P \).
We choose $D=n^{1/3}$.

\noindent \medskip \textbf{Step 2: Counting Incidences in Cells:}
Let $\tau_1,...,\tau_s$ ($s=O(D^2)$) denote the cells of the partition. We now consider the incidences between the points and lines inside these cells. Since each cell contains at most \( O(n/D^2)\) points, we can apply the weaker bound on the number of incidences in each cell.  Note that for our choice of $D$ the number of points in each cell is $O(n^{1/3})$ and the number of cells is $D^2= O(n^{2/3})$. In total, the number of incidences within the cells is bounded by:
\[
I_{\text{in cells}} = O\left( \sum^s_{i=1} \left( n^{2/3} + n_i \right) \right) = O(n^{4/3} + nD).
\]
Where $n_i$ is the number of lines intersecting the interior of cell $\tau_i$. Each line can intersect the interior of at most $D$ cells. Here we use the fact that a line can cross the zero set of a polynomial of degree $D$ in at most $D$ points (unless it is contained in its zero set in which case it does not intersect the interior of any cell). 

\noindent \medskip \textbf{Step 3: Counting Incidences on the Zero Set \( Z(f) \):}
The next step is to count the incidences that occur on the zero set \( Z(f) \), which is the algebraic curve of degree \( D \). Each line in \( L \) can intersect \( Z(f) \) in at most \( D \) points, unless it is contained in $Z(f)$ and hence the number of incidences between such lines (that are not contained in \( Z(f) \)) and the points on $Z(f)$ is bounded by \( O(nD) \). Since there are at most $D$ lines that can be contained in $Z(f)$, the number of their incidences with the points is also trivially at most $nD$. By our choice of  \( D = n^{1/3} \), in total we have: 
\[
I(P,L) = O\left(n^{4/3}\right).
\]

This concludes the proof. \qed

The asymptotically matching lower bound for point lines incidences can be constructed using the grid. Put $k=n^{1/3}$ and let $P=[k]\times[k^2]$ where $[k]=\{1,\ldots,k\}$. 
The set of lines is given by $L=\{ax+b | a\in [k/2], b\in [\frac{k^2}{2}]\}$. The number of points is $k^3$ and the number of lines is $k^3/4$. Notice that each line is incident with $k/2$ points so the total number of incidences is $k^4/8 = \Omega (n^{4/3})$.

\subsection{Extensions of the Szemerédi-Trotter Theorem}
In recent years, the Szemerédi-Trotter theorem has been extended and refined in various directions, exploring incidences in higher dimensions, with other geometric objects, and in non-Euclidean settings. This foundational result has deep connections to other fields of mathematics, including discrete geometry, algebraic geometry, graph theory, and additive combinatorics.

Below, we only mention very few results on geometric incidences. For more on incidences, see the book of Adam Sheffer \cite{Sheffer2023} and the references therein.

\medskip \noindent \textbf{Pach-Sharir theorem.}
Pach and Sharir \cite{PachSharir} extended the Szemerédi-Trotter theorem to incidences between points and curves with bounded degrees of freedom and bounded intersection numbers. Specifically, they considered a set \( P \) of \( m \) points and a set \( \mathcal{C} \) of \( n \) curves in the plane. They assumed that every pair of distinct curves in \( \mathcal{C} \) intersect at most \( t \) times and that the curves in \( \mathcal{C} \) have so-called $s$ degrees of freedom meaning that for every set of $s$ points in the plane there exists at most one curve in \( \mathcal{C} \) that contains all these points. Under these conditions, they showed that the number of incidences \( I(P, \mathcal{C}) \) between \( P \) and \(\mathcal{C}\) is bounded by

\[
I(P, \mathcal{C}) = O\left(m^{\frac{s}{2s-1}}n^{\frac{2s-2}{2s-1}} + m+n \right).
\]
where the constant of proportionality depends on $s$ and $t$.
Notice that for the special case when \( \mathcal{C} \) consists of lines, i.e., the  number of degrees of freedom $s=2$ (and also $t=2$) we recover the bound $O(m^{2/3}n^{2/3}+m+n)$ of Szemerédi and Trotter.

\medskip \noindent \textbf{Erd{\H o}s unit distance problem.}
One of the most famous open problems in combinatorial geometry, posed by Paul Erdős in 1946, concerns the maximum number of times a unit distance can occur among \( n \) points in the plane. Formally, let \( P \) be a set of \( n \) points in the Euclidean plane. Erdős asked for the maximum number \( U(n) \) of pairs \( (p,q) \) with \( p, q \in P \) such that the Euclidean distance between \( p \) and \( q \) is exactly 1, i.e., \( \|p - q\| = 1 \). Erdős conjectured that \( U(n) = O(n^{1+\epsilon}) \) for any \( \epsilon > 0 \).

The best known upper bound for \( U(n) \), due to Spencer, Szemerédi, and Trotter (1984), is \( U(n) = O(n^{4/3}) \). On the other hand, the best known lower bound, which can be achieved by placing the points in a lattice structure, is \( U(n) = \Omega(n^{1+c/\log\log n}) \) for some constant \( c \). Despite substantial progress, closing the gap between these bounds remains an open problem and a central question in discrete geometry.

For a set $P$ of points in the plane, define a graph $G=(P,E)$ on $P$ where the edges consist of exactly those pairs in $P$ of Euclidean distance $1$. Notice that such a graph cannot contain $K_{2,3}$ as a subgraph. Indeed, given three points $x,y,z$ there can be at most one point $p$ that is at exactly Euclidean distance $1$ from both $x,y$ and $z$. (In fact there can be at most one point that is equidistant to $x,y$ and $z$). This means that at most one vertex in $G$ can be a neighbor of a fixed tripple of vertices. Moreover, one can show that this problem is equivalent to bounding the maximum possible number of incidences $I(P,C)$ between a set $P$ of $n$ points in the plane and a set $C$ of $n$ unit-radius circles in the plane. Indeed, let $P$ be a set of $n$ points in the plane and let $C$ be the set $C=\{c_p|p \in P\}$ of $n$ unit circles where $c_p$ is the unit circle centered about the point $p$. It is easily seen that the number of unit distances $U(P)$ among the points of $P$ satisfies $U(P)=2I(P,C)$ since each pair $p,q$ of distance $1$ in $P$ gives rise to two incidences. $p$ is incident with $c_q$ and $q$ with $c_p$. We leave it as an exercise to the reader to show that any upper bound on the number of unit distances will give the same asymptotic upper bound on the number of incidences between $n$ points and $n$ unit circles.  

\medskip \noindent \textbf{Point-line incidences in the complex plane.}
The Szemerédi-Trotter bound has been generalized to the complex plane $\mathbb{C}^2$ by Cs. Tóth \cite{Toth15} (see also \cite{Zahl15}). Specifically, Tóth showed that the bound $I(m, n) = O(m^{2/3}n^{2/3} + m + n)$ on the maximum possible number of incidences between $m$ points and $n$ lines also holds in the complex plane by using combinatorial methods adapted to the algebraic structure of $\mathbb{C}$.

\medskip \noindent \textbf{Point-hyperplane incidences in higher dimensions.}
Consider a natural generalization of the Szemerédi-Trotter bound to higher dimensions. 
For example, a bound on the number of incidences between $n$ points and $n$ planes in $\Re^3$.
Notice that without further assumptions on the input, we cannot obtain any non-trivial bound.
Take for example a set of $n$ co-linear points (i.e., all lie on some line $\ell$) and take a set of $n$ distinct planes containing $\ell$. 
Then the number of incidences is $n^2$. However, in this example the incidence graph is a complete bipartite graph.
So, as in Zarankiewicz's problem it makes sense to add the assumption that the incidence graph does not contain some $K_{t,t}$ for fixed $t$. For $d \geq 3$ Apfelbaum and Sharir \cite{ApfelbaumS07} proved that any incidence graph between $n$ points and $n$ hyperplanes in $\Re^d$ avoiding a $K_{t,t}$ has $O_{t,d}(n^{2-\frac{2}{d+1}})$ edges. See also \cite{SudakovTomon23} for a non-trivial lower bound construction.
This higher dimensional incidence setting was studied for various other cases including points and $k$-dimensional flats in $\Re^d$ or $k$-dimensional varieties.
In the next section we discuss a far reaching generalization of such algebraic incidence graphs.

\subsection{Further improved bounds on special semi-algebraic incidence graphs}
The improved bound on the Zarankiewicz's problem for semi-algebraic graphs and the above example (of discs in the plane) raises the natural question whether in some cases a further improvement can be obtained? 
Motivated by this Basit, Chernikov, Starchenko, Tao, and Tran~\cite{BCS+21} studied incidence graphs of points and axis-parallel boxes in $\mathbb{R}^d$, under the additional assumption that they are $K_{t,t}$-free. They obtained an $O_t(n \log^{2d} n)$ bound in $\mathbb{R}^d$, and a sharp $O_t(n\frac{\log n}{\log \log n})$ bound for dyadic axis-parallel rectangles in the plane. Related results were independently obtained by Tomon and Zakharov~\cite{TZ21}.
	
Recently, Chan and Har-Peled~\cite{CH23} initiated a systematic study of Zarankiewicz's problem for incidence graphs of points and~various geometric objects. They obtained an $O_t(n(\frac{\log n}{\log \log n})^{d-1})$ bound for the incidence graph of points and axis-parallel boxes in $\mathbb{R}^d$ and observed that a matching lower bound construction appears in a classical paper of Chazelle (\cite{Chazelle90}; see also \cite{Tomon23}). They also obtained an $O_t(n \log \log n)$ bound for points and pseudo-discs in the plane, and bounds for points and halfspaces, balls, shapes with `low union complexity', and more. The proofs in~\cite{CH23} use a variety of techniques, including shallow cuttings, a geometric divide-and-conquer, and biclique covers.
Table~\ref{tab:incidences_biclique} lists some of their improved bounds.

\begin{table}[ht]
\centering
\begin{tabular}{|c|c|c|}
      \hline
      \textbf{dimension}
      &
        \textbf{geometric shapes}
      &
        \textbf{bound}
      
      \\
      \hline
      \hline
           $d>1$
      &
        $\begin{array}{c}
          \text{axis-aligned}\\
          \text{boxes}
         \end{array}$
      &

      $O\bigl( t n(\log n/\log\log n)^{d-1} +\
      km\log^{d-2+\epsilon}n\bigr)$
    
      \\
      \hline
      \hline
      
      $d=2,3$
      &
        halfplanes
      &
        $O\pth{ t(n+m) }\Bigr.$
     
      \\
      \hline
      $d>3$
      &
        halfspaces
      &
        $O\Bigl(t^{2/(\floor{d/2}+1)}
        (mn)^{\floor{d/2}/(\floor{d/2}+1)}
      + t (n+m) \Bigr)$

      \\
      \hline
      \hline
      $d=2$
      &
        disks
      &
        $O\pth{ t(n+m) }\Bigr.$
     
      \\
      \hline
      $d>3$
      &
        balls
      &
       
       $ O\Bigl(
        t^{2/(\lceil{d/2}\rceil+1)}(mn)^{\lceil{d/2}\rceil/(\lceil{d/2}\rceil+1)}
             + t (n+m) \Bigr)$

      \\
      \hline
      \hline
      $d=2$
      &
        $
        \begin{array}{c}
          \text{shapes with union}\\
          \text{complexity }\\
          U(m)
        \end{array}$
      &
        \begin{math}
            O( t n +  t U(m) (\log\log m +\log t)) \Bigr.
        \end{math}
      
      \\
      \hline
      $d=2$
      &
        pseudo-disks
      &
        \begin{math}
            O( tn + tm (\log\log m + \log t)) \Bigr.
        \end{math} [See~\ref{thm:pd-intro} for further improvements]
      
      \\
      \hline
      $d=2$
      &
        fat triangles
      &
        \begin{math}
        O( tn + tm (\log^*m)(\log^*m 
      + \log\log t)) \Bigr.
      
              \end{math}

      \\
      \hline
    \end{tabular}
    \caption{Summary of the various bounds on the number of incidences between $n$ points and $m$ geometric shapes under the condition that the incidence graph avoids $K_{t,t}$. The bounds are from ~\cite{CH23}. The function $U(m)$ describes the maximum possible \emph{union-complexity} of any $m$ objects in the corresponding family. The union-complexity of a set of objects is the number of vertices of the arrangement of the objects that belong to the boundary of the union of the objects.}
    \smallskip%
    \label{tab:incidences_biclique}

\end{table}

\section{Zarankiewicz's problem for intersection graphs of two families}
\label{sec:bipartite-intersection}
 The \emph{intersection graph} of a family $\mathcal{F}$ of geometric objects is a graph whose vertex set is $\mathcal{F}$, and whose edges connect pairs of objects whose intersection is non-empty.  In the general (i.e., non-bipartite) setting, $K_t$-free intersection graphs of geometric objects were studied extensively, and have numerous applications. One such example is the study of quasi-planar topological graphs (see, e.g.,~\cite{FoxP12}). Another example is the study of $\chi$-bounded families of graphs (see, e.g., \cite{ScottS20}). 
	
	Generalizing the systematic study of Zarankiewicz's problem for incidence graphs initiated by Basit et al. \cite{BCS+21} and by Chan and Har-Peled~\cite{CH23}, Keller and Smorodinsky study Zarankiewicz's problem for \emph{bipartite} intersection graphs of geometric objects. 
 The study of bipartite intersection graphs of geometric objects is a natural generalization of incidence graphs, in which only point-object incidences are taken into account, but not intersections between the objects. 
That is, bipartite graphs of the form $G_{A,B}=(A \cup B,E)$, where $A,B$ are families of geometric objects, and objects $x \in A, y\in B$ are connected by an edge if and only if their intersection is non-empty. Obviously, incidence graphs are the special case where $A$ consists of a set of points. 
	It is important to note that unlike the family of all graphs, this setting (i.e., bipartite intersection graphs of geometric objects) is different from the (standard) intersection graph of the family $A \cup B$, in which intersections inside $A$ and inside $B$ are also taken into account. The stark difference is exemplified in many geometric cases. One notable example is the case of two families $A,B$ of segments in the plane. Consider the tightness examples of the Szemer\'{e}di-Trotter theorem. Namely take $A$ to be a set of $n$ points and $B$ a set of $n$ lines with a total of $\Omega(n^{4/3})$ incidences. This incidence graph can be viewed as a bipartite intersection graph of two families $A$ and $B$ of segments where each point in $A$ is a degenerate segment (which we can slightly expand to make it non-degenrate) and each line in $B$ can be shortened to a segment that contains all the points of $A$ that it is incident to. Obviously this bipartite intersection graph does not contain a copy of $K_{2,2}$. So, this is an example of a bipartite intersection graph of two families $A$ and $B$ of $n$ segments avoiding a copy of $K_{2,2}$ and having $\Omega(n^{4/3})$ edges. In contrast, if the (standard) intersection graph of $A \cup B$ is $K_{2,2}$-free, then a result of Fox and Pach~\cite{FP08} implies an upper bound of $O(n)$ on its number of edges (see also~\cite{MustafaP16}). This linear upper bound holds for the more general family of string graphs discussed later in Section~\ref{sec:strings}.

\subsection{Intersection graphs of pseudo-discs.}
\label{subsec:pseudo-discs}
A \emph{family of pseudo-discs} is a family of simple closed Jordan regions in the plane such that the boundaries of any two regions intersect in at most two points. For example, a family of homothets (scaled translation copies) of a given convex body in the plane is a family of pseudo-discs. 
Note that in general pseudo-discs are not semi-algebraic graphs. Recently, Keller and Smorodinsky \cite{KellerS24} obtained a tight linear upper bound for Zarankiewicz's problem for the intersection graph of two families of pseudo-discs. This is the first non-trivial non semi-algebraic family of graphs that admit linear bound. 
	
	\begin{theorem}\label{thm:pd-intro}
		Let $t \geq 2$ and let $G=G_{A,B}$ be the bipartite intersection graph of families $A,B$ of pseudo-discs, with $|A|=|B|=n$. If $G$ is $K_{t,t}$-free then $|E(G)|=O(t^6n)$.
	\end{theorem}
	In fact, it was shown in \cite{KellerS24} that the assertion of Theorem~\ref{thm:pd-intro} holds (with a slightly weaker bound of $O(t^8n)$) for a wider class of bipartite intersection graphs of any two families of so-called \emph{non-piercing regions}. A family $\mathcal{F}$ of regions in the plane is called non-piercing if for any two regions $S,T \in \mathcal{F}$, the region $S \setminus T$ is connected.
	
	Theorem~\ref{thm:pd-intro} improves and generalizes the bound $O(n\log \log n)$ of Chan and Har-Peled~\cite{CH23} (mentioned above in Table~\ref{tab:incidences_biclique}). The bound of \cite{CH23} was proved only for the special case where $A$ is a set of points. Namely for incidence graphs of points and~pseudo-discs. 
	An interesting problem which is left open is whether the dependence on $t$ in Theorem~\ref{thm:pd-intro} can be improved. It seems that the right dependence should be linear, like in the bounds of Chan and Har-Peled~\cite{CH23}.

Next, we sketch the proof given by \cite{KellerS24} for Theorem~\ref{thm:pd-intro}. As mentioned, the proof uses a rather new approach to Zarankiewicz's problem which uses $\epsilon-t$-nets.
It was shown in \cite{KellerS24} that the primal and the dual hypergraphs of $G$ admit $\epsilon$-$t$-nets of size $O_t(1/\epsilon)$ for all $ n \geq \frac{2t}{\epsilon}$. 

  \begin{theorem}[\cite{KellerS24}] \label{thm:eps-t-pseudo_discs}
		Let $\F_1$ and $\F_2$ be two families of pseudo-discs. Let $H$ be the primal hypergraph $H_{G_{\F_1,\F_2}}$. Recall that this is the hypergraph with vertex-set  $\F_1$, where each $b \in \F_2$ defines a hyperedge $e_b=\{ a \in \F_1 : a \cap b \neq \emptyset  \}$. If $|\F_1|=n$ and $\epsilon n \geq 2t$ then $H$ admits an $\epsilon$-$t$-net of size $O(t^5 \cdot \frac{1}{\epsilon})$.
	\end{theorem}


	Using the existence of small $\epsilon-t$-nets one can prove the following lemma which provides bounds on the number of ``high"-degree vertices in intersection graphs of pseudo-discs avoiding $K_{t,t}$: 
	\begin{lemma}\label{lem:pd}
		Let $\F_1,\F_2$ be two families of pseudo-discs with $|\F_1|=|\F_2|=n$, and let $G=G_{\F_1,\F_2}$ be the bipartite intersection graph of $\F_1,\F_2$. If $G$ is $K_{t,t}$-free and $\ell \geq 2t$, then the number of vertices in $\F_1 \cup \F_2$ whose degree in $G$ is at least $\ell$ is $O(t^6 \frac{n}{\ell})$.
	\end{lemma}
	
	\begin{proof}[Proof of Lemma \ref{lem:pd}]
		To prove the lemma one can assume, without loss of generality, that all the ``heavy"  vertices (i.e., those of degree at least $\ell$) belong to $\F_2$. Let $\epsilon=\frac{\ell}{n}$. Since $\epsilon n = \ell \geq 2t$, we can apply Theorem \ref{thm:eps-t-pseudo_discs} to obtain an $\epsilon$-$t$-net $N$ of size $O(\frac{t^5}{\epsilon})$ for the primal hypergraph $H_G$. Each hyperedge of $H_G$ of size at least $\epsilon n =\ell$ contains a $t$-tuple from $N$, but since $G$ is $K_{t,t}$-free, each such a $t$-tuple participates in at most $t-1$ hyperedges.
		
		Therefore, the total number of hyperedges of size at least $\ell$ in $H_G=(\F_1,\E_{\F_2})$ is at most  $(t-1)O(\frac{t^5}{\epsilon})=O(\frac{t^6}{\epsilon})=O(\frac{t^6 n}{\ell})$. This is exactly the number of vertices in $\F_2$ with degree at least $\ell$ in $G$. This completes the proof of Lemma \ref{lem:pd}.
	\end{proof}
	
	Now we are ready to prove Theorem \ref{thm:pd-intro}. The idea of the proof is to bound the number of edges incident with some light vertices and use recursion to bound the number of edges in the remaining graph. Namely, use recursion on the graph that is induced by all ``heavy" vertices. This is done by choosing the parameter $\ell$ (in Lemma \ref{lem:pd}), such that the number of `heavy' vertices in each part of the graph is reduced by a factor of 2. 
	 This choice can be made since (unlike in the general case of Theorem \ref{thm:eps-t-net}), Theorem \ref{thm:eps-t-pseudo_discs} holds already 
	when $\epsilon n \geq 2t$. 
	
	\begin{proof}[Proof of Theorem \ref{thm:pd-intro}]
		Denote by $f(n)$ the maximum possible number of edges in $G_{\F_1,\F_2}$, for $\F_1,\F_2$ as in the statement of the theorem. Next, we show that $f(n)$ satisfies the recursion $$f(n) \leq O(n) + f(\frac{n}{2})$$
  More formally, Let $C \geq 1$ be a universal constant such that Lemma \ref{lem:pd} holds with $C \frac{t^6 n}{\ell}$. The proof is by induction that the claim holds with $f(n) \leq 8 C t^6 n$.
		
		For $n \leq 8 C t^6$, the assertion is trivial since $|E(G_{\F_1,\F_2})| \leq n^2$. We assume correctness for $\frac{n}{2}$ and prove the assertion for $n$. By Lemma \ref{lem:pd} with $\ell=2Ct^6 \geq 2t$, the number of vertices in $G_{\F_1,\F_2}$ with degree at least $2Ct^6$ is at most $C \frac{t^6 n}{\ell}=\frac{n}{2}$. 
		
		Recall that a vertex in $\F_1 \cup \F_2$ is called `heavy' if its degree in $G_{\F_1,\F_2}$ is at least $\ell$, and otherwise, it is called `light'.
		There are at most $n \ell$ edges in $G_{\F_1,\F_2}$ that connect a light vertex of $\F_1$ (resp., $\F_2$) with some vertex of $\F_2$ (resp., $\F_1$). The number of edges in $G_{\F_1,\F_2}$ that connect two heavy vertices is at most $f(\frac{n}{2})$, and by the induction hypothesis, $f(\frac{n}{2}) \leq 4Ct^6n$. Therefore, the total number of edges in $G_{\F_1,\F_2}$ is at most $(2Ct^6+2Ct^6+4Ct^6)n=8Ct^6n$. This completes the induction step.
	\end{proof}

\subsection{Intersection graphs of axis-parallel rectangles.} 
Recall that the work of Basit et al. \cite{BCS+21} mentioned above (and also of \cite{CH23,TZ21}) focus mainly on the class of incidence graphs of points and~axis-parallel rectangles (and more generally, axis-parallel boxes in $\mathbb{R}^d$). 
	
	Tomon and Zakharov~\cite{TZ21} studied the related notion of $K_{t,t}$-free intersection graphs of two families of axis-parallel boxes, under the stronger assumption that a $K_{t,t}$ including intersections inside $A$ and $B$ is also forbidden. They obtained a bound of $O(n \log^{2d+3} n)$ for $K_{t,t}$-free intersection graphs of axis-parallel boxes in $\mathbb{R}^d$, as well as a bound of $O(n \log n)$ for $K_{2,2}$-free incidence graphs of points and axis-parallel rectangles in the plane.
	
	The paper \cite{KellerS24} obtains the following tight bound for bipartite intersection graphs of axis-parallel rectangles:
	\begin{theorem}\label{thm:rect-intro}
		Let $t \geq 2$, and let $G=G_{A,B}$ be the bipartite intersection graph of families $A,B$ of axis-parallel rectangles in general position\footnote{The general position means that no two edges of rectangles in $A \cup B$ lie on the same vertical or horizontal line.}, with $|A|=|B|=n$. If $G$ is $K_{t,t}$-free, then $|E(G)|=O\left(t^6 n \frac{\log n}{\log \log n}\right)$.
	\end{theorem} 
	As follows from a lower bound given in~\cite{BCS+21}, this result is sharp in terms of the dependence on $n$ even in the special case where one of the families consists of points and the other consists of dyadic axis-parallel rectangles.
	
	An interesting open question is whether the dependence on $t$ can be improved. It seems that the `right' dependence should be linear.
	
	\medskip
	
	In the case of bipartite intersections graphs of families of axis-parallel rectangles, and even in the more basic case of incidence graphs of points and axis-parallel rectangles, the currently known bounds on the size of $\epsilon$-$t$-nets do not allow obtaining efficient bounds for Zarankiewicz's problem using the $\epsilon$-$t$-net based strategy. Indeed, among these settings, `small'-sized $\epsilon$-$t$-nets for all $\epsilon \geq c/n$ are known to exist only for the incidence hypergraph of points and axis-parallel rectangles, and the size of the $\epsilon$-$t$-net is $O(\frac{1}{\epsilon}\log{\frac{1}{\epsilon}} \log \log{\frac{1}{\epsilon}})$ (see~\cite[Theorem~6.10]{AJKSY22}). Applying the strategy described above with an $\epsilon$-$t$-net of such size would lead to an upper bound on the number of edges in a $K_{t,t}$-free incidence graph of points and axis-parallel rectangles, that is no better than $O(n\log n \log \log n)$.
	
	In order to obtain the stronger (and tight) bound of $O(n\frac{\log n}{\log \log n})$ in the more general setting of bipartite intersection graphs of two families of axis-parallel rectangles, the technique used in \cite{KellerS24} combines the result of Chan and Har-Peled~\cite{CH23} with a combinatorial argument. This yields an $O(t^6 n)$ upper bound on the number of edges in the bipartite intersection graph of two families of $n$ axis-parallel \emph{frames} (i.e., boundaries of rectangles) in the plane, and an improved bound of $O(t^4n)$ on the number of edges in the intersection graph of points and pseudo-discs (for which Chan and Har-Peled~\cite{CH23} obtained the bound $O(n\log \log n)$). 
	
\subsection{String Graphs}
\label{sec:strings}
A \textit{string} is a continuous curve in the plane.
 A \textit{string graph} is the intersection graph of a collection strings. The class of string graphs is a rich and diverse class of graphs with numerous connections to other fields in mathematics and computer science. 
 String graphs have several interesting properties and arise in many areas of computational geometry and graph theory. In particular, string graphs generalize many well-known graph families. For instance, \textit{interval graphs} are a subclass of string graphs where the curves are line segments on a single axis. Planar graphs are also a sub-class of string graphs. However, in general, string graphs are not necessarily planar. For example, the complete graph $K_n$ is a string graph for every $n$ but it is well known that already already $K_5$ is not planar.

Let $\cal S$ denote the family of all string graphs.
Let us briefly discuss Zarankiewicz's problem for string graphs.

A \textit{separator} for a graph $G=(V,E)$ is a subset $S \subset V$ such that every connected component of $G[V\setminus S]$ contains at most $\frac{2}{3}|V|$ vertices. The size of the separator is the cardinality of $|S|$ of $S$.
Lee \cite{Lee17} proved the following optimal separator theorem which could be viewed as an extension of the famous Lipton-Tarjan theorem for planar graphs (see also  \cite{FP08}). 
\begin{theorem}\cite{Lee17}
\label{thm:separator-strings}
Let $G=(V,E)$ be a string graph. Then there exists a separator for $G$ of size $\sqrt{|E|}$.
\end{theorem}
It was shown in \cite{FP14} that Theorem~\ref{thm:separator-strings} implies the following linear upper bound
on Zarankiewicz's problem for string graphs:
\begin{theorem}[\cite{FP14}]
\label{thm:Zarankiewicz-strings}
    \end{theorem}
\[
ex_{\cal S}(n,K_{t,t}) = O_t( n)
\]

Let us now describe a very simple proof of a linear bound $O_t(n)$ for the above setting of string graphs:
Hunter et al. \cite{HMST24} observed that it is in fact, an easy corollary of Thoerem~\ref{thm:induced-proper-subdivision} albeit with a much worse dependency on $t$ in the big-`O'  notation. 

\begin{proof} (of a linear bound on $ex_{\cal S}(n,K_{t,t})$ \cite{HMST24})
    It is well known and easy to verify that if every edge of a given graph $H$ is subdivided, the resulting graph is a string graph if and only if $H$ is planar. In particular, a proper subdivision of the complete graph $K_5$ is not a string graph, because $K_5$ is not planar. Note also that any induced copy of a given graph $H$ in a string graph $G$ implies that $H$ is also a string graph. So any string graph does not contain an induced copy of a proper subdivision of $K_5$. Therefore, Theorem~\ref{thm:induced-proper-subdivision} implies that $ex_{\cal S}(n,K_{t,t}) =O_t(n)$
    where the constant in the big-`O'  notation is roughly $t^{12500}$.
\end{proof}

\subsection{Intersection graphs of two families of strings}
Let us now focus on the setting of intersection graphs of two families of strings. Namely let $G=G_{A,B}$
where $A$ and $B$ are families of strings. Notice that, as mentioned above, this is not necessarily a string graph since we ignore the intersections of the strings within the same family. Such graphs exhibit a much different behavior in terms of Zarankiewicz's problem as is witnessed by the observation
that any point-line incidence graph in the plane is a special case of such graphs. Indeed, given a set $A$ of $n$ points and a set $B$ of $n$ lines, one can view the points (and the lines) as strings and the incidence graph as the graph $G_{A,B}$. We know that such graphs do not contain a $K_{2,2}$ and as mentioned in Section~\ref{sec:incidences} there are such graphs with $\Omega(n^{4/3})$ edges. This should be contrasted with the case that the intersection graph of $A \cup B$ does not contain a $K_{t,t}$ which is a much stronger hypothesis and indeed such a graph has only $O(n)$ many edges. In fact, a graph $G_{A,B}$ where $A$ and $B$ are two families of strings could be viewed as yet another extension of incidence graphs where the points can be replaced with curves.  This leads to the following natural problem:
\begin{problem}
    Let $\cal G$ be the family of all graphs of the form $G=G_{A,B}$ where $A$ and $B$ are family of strings. Fix a constant integer $t \geq 2$.
    Obtain sharp asymptotic bounds on $ex_{\cal G}(n,K_{t,t})$.
\end{problem}

As observed, we know that $ex_{\cal G}(n,K_{t,t}) = \Omega(n^{4/3})$ already for $t=2$.

 \section{Zarankiewicz problem for hypergraphs}

 Let us consider the natural generalization of Zarankiewicz's problem in $r$-uniform hypergraphs. These are hypergraphs where each hyperedge has cardianlity $r$.
 For a given $r$-uniform $H$ Let $ex^r(n,H)$ denote the maximum number of hyperedges that an $r$-uniform hypergraph on $n$ vertices can have if it does not contain a (not necessarily induced) copy isomorphic to $H$.
 
 Let $K^r_{t,t,\ldots,t}$
 denote the complete $r$-partite $r$-uniform hypergraph with all parts having size $t$.
 
 \begin{problem}
 Obtain sharp asymptotic bounds on  $ex^r(n,K^r_{t,t,\ldots,t})$.
 \end{problem}

 In 1964 Erd\H{o}s~\cite{Erdos64} obtained the following  generalization of Theorem~\ref{thm:KST}:
 \begin{theorem}[\cite{Erdos64}]
 $$ex^r(n,K^r_{t,t,\ldots,t}) = O(n^{r-\frac{1}{t^{r-1}}}) $$
 \end{theorem}
 
 He also showed that this bound is not far from being optimal for general hypergraphs. 
 Namely, he showed that there exists an absolute constant $C$ (independent of $n,t,r$) such that $$ex^r(n,K^r_{t,t,\ldots,t}) \geq n^{r-\frac{C}{t^{r-1}}} $$
  
 This version of Zarankiewicz problem seems much harder already for $r=3$ even for natural geometrically defined hypergraphs.
Recently, improved bounds were obtained for several \emph{$r$-partite} hypergraphs that arise in geometry~\cite{BCS+21,Do18,Do19,MustafaP16}.
For example Mustafa and Pach generalized the result of Fox and Pach \cite{FP08} for any $r$-uniform intersection hypergraph of $r-1$-dimensional simplices in $\Re^r$. They prove:

\begin{theorem}[\cite{MustafaP16}]
    Let $d,t \geq 2$ be integers and let $S$ be a set of $n$ $d-1$-dimensional simplices in $\Re^d$ and let 
    $H$ be there $d$-uniform intersection hypergraph. Namely, $H$ consists of all unordered $d$ tuples $\{s_1,...,s_d\} \subset S$ with $s_1\cap...\cap s_d \neq \emptyset$. If $H$ does not contain $K^d_{t,t,\ldots,t}$ then its number of hyperedges is bounded by $O(n^{d-1+\epsilon})$ for any $\epsilon > 0$.
\end{theorem}

Note that the case $d=2$ is the special setting of $K_{t,t}$-free intersection graph of line segments which was proved by Fox and Pach \cite{FP08} to have $O(n)$ edges, a bound that they later generalized to arbitrary string graphs \cite{FP10,FP14}.
Note also that the above bound is near-optimal. We leave it as an exercise to see that there are such hypergraphs avoiding $K^d_{t,t,\ldots,t}$ with $\Omega(n^{d-1})$ hyperedges.

 \medskip \noindent \textbf{Semi-algebraic hypergraphs.}
Do~\cite{Do18} extended the result on semi-algebraic graphs (that was mentioned in Section~\ref{sec:semi-algebraic}) of Fox et al.~\cite{FPSSZ17} to $r$-uniform $r$-partite hypergraphs:
Fix some positive integers $d,t$. Let $H$ be an $r$-uniform $r$-partite hypergraph $H = (P_1 \cup...\cup P_r, E)$ where  $P_i$ is a set of points in $\Re^d$ for $i =1,\ldots,r$.
There is a natural way to generalize the notion of a semi-algebraic graph with bounded description complexity $s$ (that was discussed in Section~\ref{sec:semi-algebraic}) to $r$-uniform $r$-partite hypergraphs by simply requiring that the underlying polynomials are $r\cdot d$-variate polynomials. Do proved that if such hypergraphs are $K^r_{t,t,\ldots,t}$-free then the number of hyperedges is bounded by $O_{r,t,s}(n^{r-\frac{r}{(r-1)d+1}+\epsilon})$ for every $\epsilon > 0$. This bound is asymptotically much smaller than the general bound of $O(n^{r-\frac{1}{t^{r-1}}})$ given by Erd{\H o}s for any $t$ satisfying $\frac{r}{(r-1)d+1} > \frac{1}{t^{r-1}}$. Very recently in \cite{tidor2024} these bounds for semi-algebraic hypergraphs were improved by getting rid of the $\epsilon$ in the exponent of the big-`O'  notation and also making the dependency on $t$ and the description complexity $s$ explicit and polynomial. See also the recent paper of Rubin \cite{rubin2024} for related results.

 Let us consider the natural generalization of the setting discussed in Section~\ref{sec:bipartite-intersection}. That is the $r$-partite $r$-uniform version of intersection hypergraphs in geometry:

 \begin{problem}
 \label{prob:rpartite-intersections}
Let $\F$ be a family of shapes in $\Re^d$ (e.g., the family of all balls, all axis-parallel boxes etc).
    Let $S_1,S_2,\ldots,S_r$ be $r$ finite subsets of $\F$ each of cardinality $n$. Let $H=(S_1 \cup... \cup S_r,E)$ be the $r$-uniform $r$-partite hypergraph on the vertex set $S_1 \cup... \cup S_r$ where an $r$ tuple  $\{s_1,...,s_r\} \in S_1\times \cdot\cdot\cdot \times S_r$ is a hyperedge (i.e., $\{s_1,...,s_r\} \in E$) if and only if $s_1\cap...\cap s_r \neq \emptyset$. How many hyperedges can $H$ have if it does not contain a copy of $K^r_{t,t,\ldots,t}$? 
\end{problem}

Consider for example the case of $\F$ being the family of all discs in the plane and $r=3$. Fix an integer $t$. So the setting is as follows: We are given $3$ families $A,B,C$ each consisting of $n$ discs in the plane. Consider the $3$-partite $3$-uniform hypergraph which hyperedges are all triples $a \in A, b \in B, c \in C$ of discs that have non-empty intersection. Assume that this hypergraph is $K^3_{t,t,t}$-free. Since a disc can be viewed as a point in $\Re^3$ and since one can define intersection of triples of discs using a constant number of quadratic polynomials in $9$ variables, the result of Do~\cite{Do18} implies that the number of such triples is bounded by $O(n^{3-\frac{3}{7}})$. We believe that the true bound in this special case should be close to quadratic.
As mentioned in Section~\ref{sec:bipartite-intersection} the case $r=2$ (i.e., intersection graph of two families of discs) has optimal $O(n)$ bound even in the more general case of pseudo-discs. 

\section{Conclusion and Open Problems}
The Zarankiewicz problem, while seemingly simple to state, has proven to be a rich and challenging area of research. Its applications extend to various fields, including incidence geometry, graph theory, probability and algebraic and combinatorial geometry.

We have reviewed the general Zarankiewicz problem and its special cases that arise in geometric settings such as in the Szemerédi-Trotter theorem and incidence geometry. Key concepts such as the crossing lemma, bounded VC-dimesion, shallow cuttings, $\epsilon-t$-nets, and polynomial partitioning have played pivotal roles in addressing these problems.

Despite significant progress, the Zarankiewicz problem remains open in many cases. Future research directions could focus on improving bounds for specific geometric configurations and extending them to $r$-uniform hypergraphs, exploring connections to other areas of mathematics, and developing new techniques to tackle this challenging problem.
Below, we highlight several  problems that could further our understanding of extremal properties of geometric graphs.

\begin{enumerate}
    \item \textbf{Sharper Bounds for Graphs with Bounded VC-dimension.}  
    Theorem~\ref{thm:main-boundedvc} ( Fox et al.\cite{FPSSZ17}) provides an upper bound of $O_t(n^{2-1/d})$ for $K_{t,t}$-free bipartite graphs with VC-dimension at most $d$. Can this bound be improved to $O_t(n^{2-\frac{1}{d-\epsilon_d}})$ for some $\epsilon_d > 0$?

    \item \textbf{Intersection Graphs in Higher Dimensions.}  
    What is the correct asymptotic upper bound for bipartite intersection graphs of geometric objects in $\mathbb{R}^d$ that avoid $K_{t,t}$ for $d \geq 3$? Recent results provide improved bounds, but the tight dependence on $t$ and $d$ remains unknown.

    \item \textbf{Hunter-Milojević-Sudakov-Tomon Conjecture.}  
    Conjecture~\ref{conj:Sudakovetal} asserts that for any bipartite graph $H$, there exists a function $f_H(t)$ such that any graph with no induced copy of $H$ and avoiding $K_{t,t}$ has at most $f_H(t)\cdot\mathrm{ex}(n,H)$ edges. Can this conjecture be verified for additional classes of bipartite graphs?
       
    \item \textbf{Dependence on $t$.}  
    Several results (e.g., Theorem~\ref{thm:pd-intro}) provide bounds with polynomial dependence on $t$. Is the dependence on $t$ always linear in the Zarankiewicz problem for geometric intersection graphs?

    \item \textbf{Zarankiewicz Problem for $r$-uniform $r$-partite geometric intersection Hypergraphs.} 
    Can the known bounds on the Zarankiewicz bounds for $r$-uniform $r$-partite geometric intersection hypergraphs of ``nice" objects (such as balls) in $\Re^d$ and avoiding $K^r_{t,t,\ldots,t}$ can be improved to near $O_{t,d}(n^{r-1})$? This is open already for discs in the plane (i.e., $d=2$) and $r=3$. The reader can easily verify the trivial lower bound of $\Omega(n^{r-1})$.
\end{enumerate}

Further progress on these problems would significantly advance the study of extremal combinatorics in geometric settings and shed new light on the interplay between graph theory and discrete geometry.

\section*{Acknowledgments}
I would like to express my sincere gratitude to Chaya Keller for her invaluable contributions to this paper. Her careful reading, insightful comments, identification of relevant references, and provision of valuable feedback on specific sections were essential in shaping this paper into its final form. I would also like to sincerely thank the anonymous referees for their invaluable feedback and insightful comments, which have greatly contributed to improving the presentation and clarity of this paper.
\bibliographystyle{plain}
\bibliography{references-klan}
		
\end{document}